\documentclass[11pt]{amsart}
\usepackage[margin=1.1in]{geometry}
\usepackage{amsmath,amssymb,amsthm}
\usepackage{booktabs,array}
\usepackage{enumitem}
\usepackage[hidelinks]{hyperref}

\newtheorem{theorem}{Theorem}[section]
\newtheorem{lemma}[theorem]{Lemma}
\newtheorem{proposition}[theorem]{Proposition}
\newtheorem{corollary}[theorem]{Corollary}
\newtheorem{conjecture}[theorem]{Conjecture}
\newtheorem*{theoremA}{Theorem A}
\newtheorem*{theoremB}{Theorem B}
\theoremstyle{definition}
\newtheorem{definition}[theorem]{Definition}
\newtheorem{example}[theorem]{Example}
\newtheorem{remark}[theorem]{Remark}
\newtheorem{question}[theorem]{Question}

\makeatletter
\@namedef{subjclassname@2020}{\textup{2020} Mathematics Subject Classification}
\makeatother

\newcommand{\Z}{\mathbb{Z}}
\newcommand{\ab}{\mathrm{ab}}

\newcommand{\Dic}{\mathrm{Dic}}
\newcommand{\mult}[1]{\{\!\{#1\}\!\}}   

\title{On groups with an adjacent product property}
\author{Mohammad Javaheri}
\address{Department of Mathematics, Siena University, 515 Loudon Road, Loudonville, NY 12211, USA}
\email{mjavaheri@siena.edu}
\date{Revised, November 2026}
\subjclass[2020]{05E16, 20D60, 05B15}
\keywords{Harmonious groups, R-harmonious groups, sequenceable groups, complete mappings}

\begin{document}

\begin{abstract}
A finite group $G$ of order $n$ is \emph{R-harmonious} if the non-identity elements of $G$ can be listed as $g_1,g_2,\dots,g_{n-1}$ so that the consecutive products $g_1g_2,g_2g_3,\dots,g_{n-1}g_1$ again list every non-identity element exactly once. Beals, Gallian, Headley and Jungreis proved that an abelian group is R-harmonious if and only if it does not have a unique element of order~$2$ and is not $\Z_3$. We prove a lifting lemma which shows that, under suitable conditions on $G/H$, a group $G$ with a normal subgroup $H$ of odd order is R-harmonious. As applications we show that every group of odd order whose abelianization is not an elementary abelian $3$-group is R-harmonious; in particular, every nontrivial group of odd order not divisible by~$3$ is R-harmonious. We also show that every split extension of a group of odd order by a ``harmoniously matched'' group is harmoniously matched, and hence R-harmonious. Together with a computer search, this gives many non-abelian even-order R-harmonious groups, among them the dihedral groups $D_{2n}$ with $n\equiv 8\pmod{16}$ or $n\equiv 16\pmod{32}$. Finally, we observe that the quaternion group $Q_8$ satisfies the Hall--Paige condition but is not R-harmonious.
\end{abstract}

\maketitle

\section{Introduction}

A \emph{harmonious sequence} in a finite group $G$ of order $n$ is an ordering $g_0,g_1,\dots,g_{n-1}$ of the elements of $G$ such that the consecutive products $g_0g_1,g_1g_2,\dots,g_{n-1}g_0$ also form an ordering of the elements of $G$. Harmonious sequences were introduced by Beals, Gallian, Headley and Jungreis~\cite{BGHJ} in connection with \emph{complete mappings}: bijections $\varphi\colon G\to G$ for which $x\mapsto x\varphi(x)$ is also a bijection~\cite{Paige}. If $g_0,\dots,g_{n-1}$ is harmonious, then $\varphi(g_i)=g_{i+1}$ (indices mod $n$) is a complete mapping. By the Hall--Paige theorem~\cite{HP,Wilcox,Evans,BCCSZ}, a finite group admits a complete mapping if and only if its Sylow $2$-subgroups are trivial or non-cyclic (the \emph{Hall--Paige condition}). Hence every harmonious group satisfies the Hall--Paige condition.

For abelian groups the Hall--Paige condition says that $G$ is \emph{non-binary}, that is, $G$ does not have a unique element of order~$2$. The elementary abelian $2$-groups $\Z_2^r$ ($r\ge 1$) are not harmonious, since in these groups two distinct adjacent terms never multiply to the identity. Beals et al.~\cite{BGHJ} proved that every abelian group that is neither binary nor an elementary abelian $2$-group is harmonious. They also proved that every group of odd order is harmonious. Some non-abelian even-order harmonious groups are known, such as the dihedral and dicyclic groups of order $8n$, $n>1$~\cite{BGHJ,Wang,WangLeonard}. No formal conjecture has been published, but it is widely expected that every group satisfying the Hall--Paige condition, other than the elementary abelian $2$-groups and finitely many small exceptions, is harmonious. M\"uyesser and Pokrovskiy~\cite{MP} proved this for all sufficiently large groups.

Beals et al.\ also considered the variant in which the identity is removed; they called a group $G$ \emph{$\#$-harmonious} if $G\setminus\{1_G\}$ admits such an ordering. By analogy with the established notion of R-sequenceability~\cite{Ollis,FGM,AKP}, which concerns consecutive quotients rather than products, we use the term \emph{R-harmonious}.

\begin{definition}\label{def:Rharm}
A finite group $G$ of order $n$ is \emph{R-harmonious} if the non-identity elements of $G$ can be ordered as $g_1,g_2,\dots,g_{n-1}$ so that $g_1g_2,\,g_2g_3,\,\dots,\,g_{n-1}g_1$ is also an ordering of the non-identity elements of~$G$. Such an ordering is called an \emph{R-harmonious sequence}.
\end{definition}

An R-harmonious sequence also gives a complete mapping: set $\varphi(1)=1$ and $\varphi(g_i)=g_{i+1}$ (indices mod $n-1$). So R-harmonious groups satisfy the Hall--Paige condition too. Unlike the harmonious case, the non-cyclic elementary abelian $2$-groups are R-harmonious. Beals et al.~\cite{BGHJ} proved that an abelian group is R-harmonious if and only if it is non-binary and is not $\Z_3$.

\medskip\noindent\textbf{Main results.}
The basic tool is a \emph{lifting lemma} (Lemma~\ref{lem:lifting}). It produces harmonious and R-harmonious sequences in a group $G$ from compatible data in a quotient $G/H$, where $H$ is any normal subgroup of odd order. The data must satisfy certain exact identities in $G$, not merely in $G/H$; Remark~\ref{rem:integers} explains why this matters. The lemma gives a unified treatment of cyclic quotients (Section~\ref{sec:cyclic}) and of split extensions (Section~\ref{sec:split}).

\begin{theoremA}[Theorem~\ref{thm:odd} and Corollary~\ref{cor:odd}]
Let $G$ be a nontrivial group of odd order whose abelianization $G^{\ab}=G/[G,G]$ is not an elementary abelian $3$-group. Then $G$ is harmoniously matched (Definition~\ref{def:matched}); in particular $G$ is R-harmonious. This applies, for instance, to every nontrivial group of odd order not divisible by $3$, and to every group of odd order with a nontrivial quotient of order prime to~$3$.
\end{theoremA}

\begin{theoremB}[Theorem~\ref{thm:split}]
Let $N$ be a normal subgroup of odd order of a finite group $G$ with a complement $K$, so that $G=K\ltimes N$. If $K$ is harmoniously matched, then so is $G$. In particular, $G$ is R-harmonious.
\end{theoremB}

Theorem~B applies to every abelian $K$ that is non-binary, not an elementary abelian $2$-group and not~$\Z_3$. A computer search shows that many small non-abelian groups are harmoniously matched (Table~\ref{tab:small}); Theorem~B then applies to these as well. Section~\ref{sec:small} records the small groups that fail to be R-harmonious. These are $\Z_3$ and $D_8$, as already known, and also the quaternion group $Q_8$, which satisfies the Hall--Paige condition. This leads to Conjecture~\ref{conj:main}.

\medskip\noindent\textbf{Notation.} We write $K\ltimes N$ for a semidirect product in which $N$ is normal. For a sequence $c_0,\dots,c_{\ell-1}$ we always read indices cyclically, so $c_\ell=c_0$. Double braces $\mult{\cdots}$ denote multisets.

\section{Harmoniously matched groups and the lifting lemma}\label{sec:lifting}

\begin{definition}\label{def:matched}
A finite group $G$ of order $n$ is \emph{harmoniously matched} if there are a harmonious sequence $g_0,g_1,\dots,g_{n-1}$ and an R-harmonious sequence $g_1',g_2',\dots,g_{n-1}'$ in $G$ such that
\[
 g_0g_1\cdots g_{n-1}=1_G,\qquad g_1'=g_0,\qquad g_{n-1}'=g_{n-1}.
\]
We then call the pair of sequences a \emph{matched pair}.
\end{definition}

We need a special kind of harmonious sequence in groups of odd order.

\begin{definition}
A harmonious sequence $h_0,h_1,\dots,h_{m-1}$ in a group $H$ of order $m$ is \emph{symmetric} if $h_ih_{m-i}=1_H$ for $1\le i\le m-1$.
\end{definition}

If $m$ is odd, the indices $1,\dots,m-1$ fall into the pairs $\{i,m-i\}$, so a symmetric sequence must have $h_0=1_H$.

\begin{proposition}[{\cite{JavSym}}]\label{prop:sym}
A finite group is symmetric harmonious if and only if it has odd order. In particular, every group $H$ of odd order $m$ has a harmonious sequence $1_H=h_0,h_1,\dots,h_{m-1}$ with $h_{m-i}=h_i^{-1}$ for $1\le i\le m-1$.
\end{proposition}

\begin{lemma}\label{lem:symprod}
Let $h_0,\dots,h_{m-1}$ be a symmetric harmonious sequence in a group $H$ of odd order $m$, and let $e$ be any integer. Then $h_1^{e}h_2^{e}\cdots h_{m-1}^{e}=1_H$.
\end{lemma}

\begin{proof}
Put $t=(m-1)/2$. Then $h_{t+1}=h_t^{-1}$, $h_{t+2}=h_{t-1}^{-1}$, and so on. The product therefore collapses from the middle: $h_1^e\cdots h_t^e\,h_t^{-e}\cdots h_1^{-e}=1_H$.
\end{proof}

The following lemma is the main tool of the paper.

\begin{lemma}[Lifting lemma]\label{lem:lifting}
Let $H$ be a normal subgroup of odd order $m$ of a finite group $G$, and let $n=[G:H]\ge 3$. Suppose that $a_0,\dots,a_{n-1}\in G$ and $a_1',\dots,a_{n-1}'\in G$ satisfy:
\begin{enumerate}[label=\textup{(L\arabic*)}]
\item\label{L1} the cosets $Ha_0,\dots,Ha_{n-1}$ are distinct, and the cosets $Ha_0a_1,Ha_1a_2,\dots,Ha_{n-1}a_0$ are distinct;
\item\label{L2} $a_0a_1\cdots a_{n-1}=1_G$;
\item\label{L3} the unique term $a_s$ lying in $H$ is equal to $1_G$;
\item\label{L4} $a_1',\dots,a_{n-1}'$ is an ordering of the terms $a_r$, $r\neq s$, with $a_1'=a_0$ and $a_{n-1}'=a_{n-1}$;
\item\label{L5} the multisets of consecutive products satisfy
\[
 \mult{a_ra_{r+1}:0\le r\le n-1}=\mult{1_G}\uplus\mult{a_i'a_{i+1}':1\le i\le n-1},
\]
where $a_n=a_0$ and $a_n'=a_1'$.
\end{enumerate}
Then $G$ is harmoniously matched.
\end{lemma}

\begin{proof}
By Proposition~\ref{prop:sym}, $H$ has a symmetric harmonious sequence $h_0=1,h_1,\dots,h_{m-1}$. (If $m=1$, this is just $h_0=1$.) Set $\sigma_{-1}=1$ and $\sigma_r=a_0a_1\cdots a_r$ for $0\le r\le n-1$; by~\ref{L2}, $\sigma_{n-1}=1$. For $0\le p\le m-1$ and $0\le r\le n-1$ define
\[
 g_{p,r}=\sigma_{r-1}^{-1}\,h_p\,\sigma_r=\bigl(\sigma_{r-1}^{-1}h_p\sigma_{r-1}\bigr)\,a_r .
\]
Then $g_{0,r}=a_r$. Since conjugation by $\sigma_{r-1}$ permutes $H$, for each fixed $r$ the map $p\mapsto g_{p,r}$ is a bijection from $\{0,\dots,m-1\}$ onto the coset $Ha_r$. By~\ref{L1}, the $mn$ elements $g_{p,r}$ are therefore exactly the elements of $G$. Let $\mathbf g$ be the sequence
\[
 \mathbf g:\ g_{0,0},\dots,g_{0,n-1},\ g_{1,0},\dots,g_{1,n-1},\ \dots,\ g_{m-1,0},\dots,g_{m-1,n-1},
\]
read cyclically, so that $g_{m-1,n-1}$ is followed by $g_{0,0}$. We say that a consecutive product $g_{p,r}g_{p,r+1}$ ($r\le n-2$), or $g_{p,n-1}g_{p+1,0}$ with $p+1$ read mod $m$, sits in \emph{position} $r$, respectively $n-1$. Write $h_m=h_0=1$.
\begin{itemize}
\item For $0\le r\le n-2$:\quad $g_{p,r}g_{p,r+1}=\sigma_{r-1}^{-1}h_p^2\sigma_{r+1}=\bigl(\sigma_{r-1}^{-1}h_p^2\sigma_{r-1}\bigr)\,a_ra_{r+1}$.
\item For $r=n-1$: since $\sigma_{n-1}=1$ we have $\sigma_{n-2}^{-1}=a_{n-1}$, so
 $g_{p,n-1}g_{p+1,0}=\sigma_{n-2}^{-1}h_p\,h_{p+1}\sigma_0=\bigl(a_{n-1}h_ph_{p+1}a_{n-1}^{-1}\bigr)\,a_{n-1}a_0$.
\end{itemize}
Thus the $m$ products in position $r$ lie in the coset $Ha_ra_{r+1}$. By~\ref{L1}, different positions give different cosets. In a position $r\le n-2$ the products are distinct, because $x\mapsto x^2$ is a bijection of the odd-order group $H$. In position $n-1$ the products are distinct because $h_0h_1,h_1h_2,\dots,h_{m-1}h_0$ are distinct. Hence the $mn$ consecutive products of $\mathbf g$ are distinct, and $\mathbf g$ is harmonious. Its product is
\[
 \prod_{p=0}^{m-1}\ \prod_{r=0}^{n-1}\sigma_{r-1}^{-1}h_p\sigma_r=\prod_{p=0}^{m-1}\sigma_{-1}^{-1}h_p^{\,n}\sigma_{n-1}=h_1^nh_2^n\cdots h_{m-1}^n=1
\]
by Lemma~\ref{lem:symprod}.

Now let $\mathbf g'$ be the sequence obtained from $\mathbf g$ by replacing the first block $a_0,\dots,a_{n-1}$ with $a_1',\dots,a_{n-1}'$. By~\ref{L3} and~\ref{L4}, the terms of $\mathbf g'$ are exactly the elements of $G\setminus\{1\}$. By~\ref{L4}, the first and last terms of $\mathbf g'$ coincide with those of $\mathbf g$, namely $a_0$ and $g_{m-1,n-1}$. So the only consecutive products that change are those inside the first block. The products $a_{n-1}g_{1,0}$ (or $a_{n-1}a_0$ when $m=1$) and $g_{m-1,n-1}a_0$ are untouched. Inside the first block, the products $a_ra_{r+1}$ for $0\le r\le n-2$ are replaced by $a_i'a_{i+1}'$ for $1\le i\le n-2$. Since $a_{n-1}'a_1'=a_{n-1}a_0$, condition~\ref{L5} gives
\[
 \mult{a_i'a_{i+1}':1\le i\le n-2}=\mult{a_ra_{r+1}:0\le r\le n-2}\setminus\mult{1}.
\]
Here $1$ does occur among $a_ra_{r+1}$ with $r\le n-2$. Otherwise $a_{n-1}a_0=1$, and then by~\ref{L5} the element $1$ would occur twice among the products $a_ra_{r+1}$, which lie in distinct cosets. Hence the consecutive products of $\mathbf g'$ are those of $\mathbf g$ with $1$ removed, that is, $G\setminus\{1\}$. So $\mathbf g'$ is R-harmonious, and $(\mathbf g,\mathbf g')$ is a matched pair.
\end{proof}

\begin{remark}\label{rem:product}
With the notation of the proof, the product of the terms of $\mathbf g'$ equals $a_1'a_2'\cdots a_{n-1}'\cdot h_1^n\cdots h_{m-1}^n=a_1'\cdots a_{n-1}'$. So if the input satisfies $a_1'\cdots a_{n-1}'=1$, then $\mathbf g'$ also has product~$1$.
\end{remark}

\begin{remark}\label{rem:integers}
Condition~\ref{L5} is an equality of elements of $G$, not of cosets. Equality in $G/H$ does not suffice. For example, the sequences $1,-1,-3,4,2,0,-2,-4,3$ and $1,-4,2,-3,-1,4,-2,3$ form a matched pair in the group $\Z_9$. However, the sums $-6$ and $6$ of consecutive terms of the first sequence are matched only modulo~$9$, by the sums $3$ and $-3$ of the second. Feeding these sequences, as integer exponents of a generator, into the construction below for the non-split group $G=\Z_{19}\rtimes\Z_{27}$ (in which an element of order $27$ acts by an automorphism of order~$9$, and $H$ is the subgroup of index~$9$) produces a sequence that is harmonious but whose companion is \emph{not} R-harmonious. This is why we work with integer sequences in Section~\ref{sec:cyclic}.
\end{remark}

\section{Cyclic quotients and harmoniously matched integer sequences}\label{sec:cyclic}

\begin{definition}\label{def:intseq}
Let $n\ge 3$ be odd. A pair of integer sequences $k_0,k_1,\dots,k_{n-1}$ and $k_1',\dots,k_{n-1}'$ is a \emph{matched integer pair of length $n$} if:
\begin{enumerate}[label=\textup{(I\arabic*)}]
\item\label{I1} $k_0,\dots,k_{n-1}$ are distinct modulo $n$, and the sums $k_0+k_1,k_1+k_2,\dots,k_{n-1}+k_0$ are distinct modulo $n$;
\item\label{I2} $k_0+k_1+\cdots+k_{n-1}=0$;
\item\label{I3} the unique term $k_s$ that is divisible by $n$ is equal to $0$;
\item\label{I4} $k_1',\dots,k_{n-1}'$ is an ordering of the terms $k_r$, $r\neq s$, with $k_1'=k_0$ and $k_{n-1}'=k_{n-1}$;
\item\label{I5} as multisets of \emph{integers},
 $\mult{k_r+k_{r+1}:0\le r\le n-1}=\mult{0}\uplus\mult{k_i'+k_{i+1}':1\le i\le n-1}$, where $k_n=k_0$ and $k_n'=k_1'$.
\end{enumerate}
\end{definition}

\begin{example}\label{ex:9}
The sequences
\[
 10,\,-10,\,-3,\,4,\,2,\,0,\,-2,\,-4,\,3 \qquad\text{and}\qquad 10,\,-4,\,2,\,-3,\,-10,\,4,\,-2,\,3
\]
form a matched integer pair of length~$9$. The consecutive sums of the first sequence form the multiset $\mult{0,\pm1,\pm2,\pm6,\pm13}$; those of the second form the same multiset with $0$ removed.
\end{example}

\begin{lemma}[Cyclic lifting]\label{lem:cyclic}
Let $H$ be a normal subgroup of odd order of a finite group $G$ such that $G/H\cong\Z_n$, and suppose that a matched integer pair of length $n$ exists. Then $G$ is harmoniously matched.
\end{lemma}

\begin{proof}
Choose $x\in G$ such that $xH$ generates $G/H$; then $Hx^a=Hx^b$ if and only if $a\equiv b\pmod n$. Put $a_r=x^{k_r}$ and $a_i'=x^{k_i'}$. Conditions \ref{I1}--\ref{I5} translate directly into \ref{L1}--\ref{L5} of Lemma~\ref{lem:lifting}. For~\ref{L3} we use that $k_s=0$ exactly, so $a_s=1$. For~\ref{L5} we use that equal integers give equal powers of~$x$; congruence modulo $n$ alone would only give equality modulo~$H$. Moreover $a_1'\cdots a_{n-1}'=x^{\sum_r k_r}=1$.
\end{proof}

\subsection{Existence of matched integer pairs}

\begin{theorem}\label{thm:intseq}
Let $n\ge5$ be odd. A matched integer pair of length $n$ exists in each of the following cases:
\begin{enumerate}[label=\textup{(\roman*)}]
\item $3\nmid n$;
\item $n=12k-3$ with $3\nmid k$, that is, $n\equiv 9$ or $21\pmod{36}$.
\end{enumerate}
\end{theorem}

The only odd $n\ge5$ not covered by Theorem~\ref{thm:intseq} are those with $n\equiv3\pmod{12}$ or $n\equiv33\pmod{36}$. For some of these, see Remark~\ref{rem:12k-3} and Appendix~\ref{app:seq}.

For $n=5$ take $k=(2,\allowbreak -2,\allowbreak -1,\allowbreak 0,\allowbreak 1)$ and $k'=(2,-1,-2,1)$; the consecutive sums are $\mult{0,\pm1,\pm3}$ and $\mult{\pm1,\pm3}$. For all other $n$ in Theorem~\ref{thm:intseq} we use one of the following three families. Note that $n\ge7$ in each of them.

\medskip\noindent\textbf{The case $n=6k+1$, $k\ge1$.} For $0\le i\le n-1$ let
\[
k_i=\begin{cases}
4k & i=0,\\
4i/3-2k-1 & 0<i\le 3k,\ i\equiv 0\pmod 3,\\
4(i-1)/3-4k & 0<i\le 3k,\ i\equiv 1\pmod 3,\\
4(i+1)/3-2 & 0<i\le 3k,\ i\equiv 2\pmod 3,\\
4i/3-8k-2 & 3k<i\le 6k,\ i\equiv 0\pmod 3,\\
4(i-1)/3-4k & 3k<i\le 6k,\ i\equiv 1\pmod 3,\\
4(i+1)/3-6k-3 & 3k<i\le 6k,\ i\equiv 2\pmod 3.
\end{cases}
\]

\noindent\textbf{The case $n=6k-1$, $k\ge2$.} For $0\le i\le n-1$ let
\[
k_i=\begin{cases}
4k-2 & i=0,\\
8i/3-2k-1 & 0<i<3k,\ i\equiv 0\pmod 3,\\
8(i-1)/3-4k+2 & 0<i<3k,\ i\equiv 1\pmod 3,\\
8(i+1)/3-6k-3 & 0<i<3k,\ i\equiv 2\pmod 3,\\
0 & i=3k,\\
8i/3-14k+1 & 3k<i\le 6k-2,\ i\equiv 0\pmod 3,\\
8(i-1)/3-10k+3 & 3k<i\le 6k-2,\ i\equiv 1\pmod 3,\\
8(i+1)/3-12k-2 & 3k<i\le 6k-2,\ i\equiv 2\pmod 3.
\end{cases}
\]

\noindent\textbf{The case $n=12k-3$, $k\ge1$, $3\nmid k$.} For $0\le i\le n-1$ let
\[
k_i=\begin{cases}
16k-6 & i=0,\\
4i-8k & 0<i\le 6k-2,\ i\equiv 0\pmod 3,\\
4i+2-16k & 0<i\le 6k-2,\ i\equiv 1\pmod 3,\\
4i+1-12k & 0<i\le 6k-2,\ i\equiv 2\pmod 3,\\
0 & i=6k-1,\\
4i+6-32k & 6k\le i\le 12k-4,\ i\equiv 0\pmod 3,\\
4i+8-40k & 6k\le i\le 12k-4,\ i\equiv 1\pmod 3,\\
4i+7-36k & 6k\le i\le 12k-4,\ i\equiv 2\pmod 3.
\end{cases}
\]
For $k=1$ this is the sequence of Example~\ref{ex:9}.

In all three cases put $\alpha=\lfloor (n+1)/4\rfloor$ and $\beta=\lfloor 3(n+1)/4\rfloor$. Define $\tau\colon\{1,\dots,n-1\}\to\{0,\dots,n-1\}$ by
\[
\tau(j)=\begin{cases}
0 & j=1,\\
n+2-2j & 2\le j\le \alpha,\\
n+1-2j & \alpha< j\le (n-1)/2,\\
2j-n & (n+1)/2\le j<\beta,\\
2j-n+1 & \beta\le j\le n-1,
\end{cases}
\]
and set $k_i'=k_{\tau(i)}$ for $1\le i\le n-1$.

\begin{proof}[Proof of Theorem~\ref{thm:intseq}]
The case $n=5$ was treated above. For the three families we check \ref{I1}--\ref{I5}. Step~(a) differs between the families $n=6k\pm1$ and $n=12k-3$; Steps (b)--(d) are uniform.

\smallskip\noindent\emph{(a) Distinctness for $n=6k\pm1$: the residues are linear.} We have $k_i\equiv c+di\pmod n$ for all $i$, where $(c,d)=(4k,\,4\cdot 3^{-1})$ if $n=6k+1$ and $(c,d)=(4k-2,\,8\cdot3^{-1})$ if $n=6k-1$. Here $3^{-1}$ is the inverse of $3$ modulo $n$. For instance, if $n=6k+1$ and $0<i\le 3k$ with $i\equiv1\pmod 3$, then $k_i-(4k+4i/3)=-8k-4/3=-\tfrac43(6k+1)\equiv0\pmod n$. Since $d$ is a unit modulo $n$, the $k_i$ are distinct modulo~$n$. The sums satisfy $k_i+k_{i+1}\equiv 2c+d(2i+1)\pmod n$ for $0\le i\le n-1$, using $dn\equiv0$ when $i=n-1$. These are distinct because $2d$ is a unit. This proves~\ref{I1}.

\smallskip\noindent\emph{(a$'$) Distinctness for $n=12k-3$: the residues are linear on each class modulo $3$.} Here $3\mid n$, so no linear formula is possible. Put
\[
 c_0=-8k,\qquad c_1=2-16k,\qquad c_2=1-12k.
\]
A direct check shows that $k_i\equiv 4i+c_\rho\pmod n$ for all $0\le i\le n$, where $\rho\in\{0,1,2\}$ is the residue of $i$ modulo $3$ and $k_n=k_0$. For instance, for $6k\le i\le 12k-4$ and $i\equiv0\pmod3$ we have $(4i+6-32k)-(4i-8k)=-2n$. Also $k_0-c_0=24k-6=2n$, and $k_n=k_0\equiv c_0\equiv 4n+c_0$, since $3\mid n$. Write $i=3q+\rho$ with $0\le q\le 4k-2$; note $n/3=4k-1$. Then
\[
 k_{3q+\rho}\equiv 12q+(4\rho+c_\rho),\qquad k_{3q+\rho}+k_{3q+\rho+1}\equiv 24q+(8\rho+4+c_\rho+c_{\rho+1})\pmod n,
\]
with $c_3=c_0$. Since $4k-1$ is odd, both $4$ and $8$ are units modulo $n/3=4k-1$. Hence, for fixed $\rho$, the values $12q$ ($0\le q\le 4k-2$) are distinct modulo $n$, and so are the values $24q$. Moreover all of them lie in $3\Z_n$. It remains to separate the three classes $\rho=0,1,2$, which lie in the cosets of $3\Z_n$ determined by the constants modulo~$3$:
\[
 4\rho+c_\rho\equiv k,\ 2k,\ 0,\qquad 8\rho+4+c_\rho+c_{\rho+1}\equiv 0,\ 2k,\ k\pmod 3\qquad(\rho=0,1,2).
\]
These residues are distinct exactly when $3\nmid k$. In that case the terms are distinct modulo $n$ and so are the consecutive sums, which proves~\ref{I1}. (If $3\mid k$, every term is divisible by $3$ and~\ref{I1} fails.)

\smallskip\noindent\emph{(b) The terms cancel in pairs.} In all three families, $k_0+k_1=0$ and $k_i+k_{n+1-i}=0$ for $2\le i\le n-1$, which gives~\ref{I2}. For example, if $n=12k-3$, $0<i\le 6k-2$ and $i\equiv0\pmod3$, then $n+1-i\ge 6k$ and $n+1-i\equiv1\pmod 3$, so $k_{n+1-i}=4(12k-2-i)+8-40k=8k-4i=-k_i$. Moreover $k_{(n+1)/2}=0$, the index being $3k+1$, $3k$ and $6k-1$ in the three families. So~\ref{I3} holds with $s=(n+1)/2$, because by~(a) or~(a$'$) no other term is divisible by~$n$.

\smallskip\noindent\emph{(c) For every odd $n\ge7$, $\tau$ is a bijection onto $\{0,\dots,n-1\}\setminus\{(n+1)/2\}$ with $\tau(1)=0$ and $\tau(n-1)=n-1$.} The four non-trivial ranges of $\tau$ produce, respectively, the odd integers in $[n+2-2\alpha,n-2]$, the even integers in $[2,n-1-2\alpha]$, the odd integers in $[1,2\beta-n-2]$, and the even integers in $[2\beta-n+1,n-1]$. If $n\equiv1\pmod4$, then $\alpha=(n-1)/4$, $\beta=(3n+1)/4$, and $(n+1)/2$ is odd. If $n\equiv3\pmod4$, then $\alpha=(n+1)/4$, $\beta=3(n+1)/4$, and $(n+1)/2$ is even. In either case the four sets partition $\{1,\dots,n-1\}\setminus\{(n+1)/2\}$. This proves~\ref{I4}.

\smallskip\noindent\emph{(d) Consecutive sums match.} For $1\le i\le n-1$ define $j(i)$ by Table~\ref{tab:j}. We claim that
\begin{equation}\label{eq:match}
 k_i'+k_{i+1}'=k_{j(i)}+k_{j(i)+1}\qquad(1\le i\le n-1,\ k_n'=k_1').
\end{equation}
The same table serves all three families. Fix a family, one of the ranges of Table~\ref{tab:j}, and a residue class of $i$ modulo $3$. Both sides of~\eqref{eq:match} are then affine functions of $(i,k)$ with rational coefficients, so~\eqref{eq:match} is a finite list of affine identities. We verify one representative case for each kind of family; the others are entirely similar. All of them were also confirmed by computer, both symbolically on each range and directly for every admissible $n<3000$. First take $n=6k+1$ and $2\le i\le \alpha-1$, so that $\tau(i)=n+2-2i$, $\tau(i+1)=n-2i$ and $j=j(i)=3k+1-2i$.
\begin{itemize}
\item If $i\equiv0\pmod3$, then $k_i'+k_{i+1}'=\tfrac43(n+2-2i)-8k-2+\tfrac43(n-2i-1)-4k=4k+2-\tfrac{16i}{3}$, and $k_j+k_{j+1}=\tfrac43(j-1)-4k+\tfrac43(j+2)-2=4k+2-\tfrac{16i}3$.
\item If $i\equiv1\pmod3$, both sides equal $6k-3-\tfrac{16(i-1)}{3}$.
\item If $i\equiv2\pmod3$, both sides equal $2k+7-\tfrac{16(i+1)}{3}$.
\end{itemize}
Next take $n=12k-3$ and $2\le i\le\alpha-1=3k-2$ with $i\equiv0\pmod3$. Then $\tau(i)=12k-1-2i\ge 6k+3$ is $\equiv2\pmod 3$, $\tau(i+1)=12k-3-2i$ is $\equiv0\pmod3$, and $j=6k-1-2i\le 6k-2$ is $\equiv2\pmod 3$. Hence
\[
k_i'+k_{i+1}'=(12k+3-8i)+(16k-6-8i)=28k-3-16i=(12k-3-8i)+(16k-8i)=k_j+k_{j+1}.
\]
The identities~\eqref{eq:match} for $n=12k-3$ hold for every $k\ge1$, including $3\mid k$; the hypothesis $3\nmid k$ enters only through~(a$'$).
Finally, $j$ is a bijection from $\{1,\dots,n-1\}$ onto $\{1,\dots,n-1\}$. Indeed, the rows of Table~\ref{tab:j} produce:
\begin{itemize}
\item from the rows $2\le i\le\alpha-1$ and $\beta\le i\le n-2$, two arithmetic progressions of difference $2$ and opposite parity, which together fill $\{2,\dots,(n-5)/2\}$;
\item from the rows $\alpha<i\le(n-3)/2$ and $(n+1)/2\le i\le\beta-2$, two such progressions filling $\{(n+3)/2,\dots,n-2\}$;
\item from the remaining rows, the values $(n-3)/2,\ 1,\ n-1$ and $\{n-2\alpha,\,2\beta-1-n\}=\{(n-1)/2,\,(n+1)/2\}$.
\end{itemize}
Since $k_0+k_1=0$, \eqref{eq:match} gives
$\mult{k_i'+k_{i+1}'}=\mult{k_j+k_{j+1}:1\le j\le n-1}=\mult{k_j+k_{j+1}:0\le j\le n-1}\setminus\mult{0}$, which is~\ref{I5}.
\end{proof}

\begin{table}[ht]
\centering
\renewcommand{\arraystretch}{1.25}
\begin{tabular}{@{}lll@{}}
\toprule
range of $i$ & $j(i)$ & $(\tau(i),\tau(i+1))$\\
\midrule
$i=1$ & $(n-3)/2$ & $(0,\,n-2)$\\
$2\le i\le \alpha-1$ & $(n+1)/2-2i$ & $(n+2-2i,\,n-2i)$\\
$i=\alpha$ & $n-2\alpha$ & $(n+2-2\alpha,\,n-1-2\alpha)$\\
$\alpha< i\le (n-3)/2$ & $(3n-1)/2-2i$ & $(n+1-2i,\,n-1-2i)$\\
$i=(n-1)/2$ & $1$ & $(2,\,1)$\\
$(n+1)/2\le i\le \beta-2$ & $2i-(n-1)/2$ & $(2i-n,\,2i+2-n)$\\
$i=\beta-1$ & $2\beta-1-n$ & $(2\beta-2-n,\,2\beta+1-n)$\\
$\beta\le i\le n-2$ & $2i-3(n-1)/2$ & $(2i-n+1,\,2i-n+3)$\\
$i=n-1$ & $n-1$ & $(n-1,\,0)$\\
\bottomrule
\end{tabular}
\medskip
\caption{The index $j(i)$ in~\eqref{eq:match}, valid for all three families. In every row $2j(i)+1\equiv\tau(i)+\tau(i+1)\pmod n$; the content of~\eqref{eq:match} is that equality holds over~$\Z$.}
\label{tab:j}
\end{table}

\begin{remark} \label{rem:12k-3}
\begin{enumerate}[label=(\alph*)]
\item By Corollary~\ref{cor:anycyclic} below, Theorem~\ref{thm:intseq} is needed only for $n=9$ and for primes $n\ge5$.
\item For the remaining lengths ($n\equiv3\pmod{12}$, $n\ge15$, or $n\equiv33\pmod{36}$) we have no general construction. The following lemma gives a simple sufficient condition, and Appendix~\ref{app:seq} lists the matched integer pairs it produces for $n=15,27,33,39$. A computer search found sequences satisfying its hypothesis for every odd $n$ with $7\le n\le 39$. We do not know whether they exist for all odd $n\ge7$ (Question~\ref{q:reversal}). For $n=3$ no matched pair exists, because $\Z_3$ is not R-harmonious.
\end{enumerate}
\end{remark}

\begin{lemma}[Reversal lemma]\label{lem:reversal}
Let $n\ge7$ be odd. Suppose that $r_0,\dots,r_{n-1}$ is a harmonious sequence in $\Z_n$ with $r_1=0$, and that for some $3\le j\le n-2$ we have $r_j=r_2-r_0$ and $r_{j+1}=r_0-r_2$. Then a matched integer pair of length $n$ exists.
\end{lemma}

\begin{proof}
Choose integers $k_0\equiv r_0$ and $k_2\equiv r_2$, and put $k_1=0$, $k_j=k_2-k_0$ and $k_{j+1}=k_0-k_2$. For the remaining indices choose any integers $k_i\equiv r_i$. The total sum is $\equiv\sum_{x\in\Z_n}x\equiv0\pmod n$. Since $n\ge7$, some index lies outside $\{0,1,2,j,j+1\}$; adjusting that term by a multiple of $n$ makes the sum $0$. Let
\[
 k'=(k_0,\ k_j,k_{j-1},\dots,k_2,\ k_{j+1},\dots,k_{n-1}).
\]
That is, delete $k_1=0$ and reverse the block $k_2,\dots,k_j$. The sums inside the reversed block are unchanged as a multiset. The sums $k_0+k_1=k_0$, $k_1+k_2=k_2$ and $k_j+k_{j+1}=0$ are replaced by $k_0+k_j=k_2$ and $k_2+k_{j+1}=k_0$. All other sums are unchanged. Hence \ref{I1}--\ref{I5} hold.
\end{proof}

\subsection{Consequences}

\begin{corollary}\label{cor:anycyclic}
Let $H$ be a normal subgroup of odd order of $G$ with $G/H$ cyclic of odd order $n>3$. Then $G$ is harmoniously matched.
\end{corollary}

\begin{proof}
Since $n>3$ is odd, either $9\mid n$ or $n$ has a prime divisor $p\ge5$. Let $q=9$ in the first case and $q=p$ in the second. There is a normal subgroup $H'\supseteq H$ with $G/H'\cong\Z_q$, and $|H'|=|H|\,n/q$ is odd. By Theorem~\ref{thm:intseq} a matched integer pair of length $q$ exists, and Lemma~\ref{lem:cyclic} applies to $H'$.
\end{proof}

\begin{example}\label{ex:nonsplit}
In Corollary~\ref{cor:anycyclic} the extension need not split. For instance, $G=\Z_{11}\rtimes\Z_{25}$, where a generator $y$ of $\Z_{25}$ acts on $\Z_{11}$ by $x\mapsto x^3$ (an automorphism of order $5$), has the normal subgroup $H=\langle x,y^5\rangle$ of order $55$ with $G/H\cong\Z_5$. Every element of $G\setminus H$ has order divisible by $25$, so $G/H$ has no complement. Nevertheless $G$ is harmoniously matched. We also checked the construction of Lemma~\ref{lem:lifting} by computer in this group and in the similar non-split groups $\Z_{29}\rtimes\Z_{49}$ and $\Z_{19}\rtimes\Z_{27}$.
\end{example}

\section{Groups of odd order}\label{sec:odd}

\begin{theorem}\label{thm:odd}
Let $G$ be a nontrivial group of odd order such that $G^{\ab}$ is not an elementary abelian $3$-group. Then $G$ is harmoniously matched, and in particular R-harmonious.
\end{theorem}

\begin{proof}
The group $G^{\ab}$ is abelian of odd order and is not an elementary abelian $3$-group (note that $G^{\ab}\neq1$, since $G$ is solvable by the Feit--Thompson theorem~\cite{FT}). Either $|G^{\ab}|$ has a prime divisor $p\ge5$, or $G^{\ab}$ is a $3$-group of exponent at least $9$. In both cases $G^{\ab}$, and hence $G$, has a cyclic quotient of order $p\ge 5$ or $9$. Apply Corollary~\ref{cor:anycyclic}.
\end{proof}

For a group $G$ of odd order, the following are equivalent: (i) $G^{\ab}$ is not an elementary abelian $3$-group; (ii) $G$ has a cyclic quotient of order greater than~$3$; (iii) $G$ has a quotient isomorphic to $\Z_p$ for a prime $p\ge5$, or to $\Z_9$.

\begin{corollary}\label{cor:odd}
Let $G$ be a nontrivial group of odd order. Then $G$ is harmoniously matched in each of the following cases:
\begin{enumerate}[label=(\alph*)]
\item $3\nmid|G|$;
\item $G$ has a nontrivial quotient of order prime to $3$;
\item $G$ has a proper normal Sylow $3$-subgroup.
\end{enumerate}
\end{corollary}

\begin{proof}
Case (a) is the special case $Q=G$ of (b). Case (c) is a special case of (b): if $P$ is a proper normal Sylow $3$-subgroup, then $Q=G/P$ is nontrivial of order prime to~$3$. For (b), let $Q\ne1$ be a quotient of $G$ with $3\nmid|Q|$. By the Feit--Thompson theorem, $Q$ is solvable, so $Q^{\ab}\neq1$. Since $Q^{\ab}$ is a quotient of $G^{\ab}$ and $3\nmid |Q^{\ab}|$, the group $G^{\ab}$ is not a $3$-group. Now apply Theorem~\ref{thm:odd}.
\end{proof}

The groups of odd order not covered by Theorem~\ref{thm:odd} are those with $G^{\ab}\cong\Z_3^{\,r}$ for some $r\geq1$. Some of them are handled by the split-extension results of the next section; see Corollary~\ref{cor:odd3}.

\section{Split extensions}\label{sec:split}

\begin{theorem}\label{thm:split}
Let $N$ be a normal subgroup of odd order of a finite group $G$ with a complement $K$, so that $G=K\ltimes N$. If $K$ is harmoniously matched, then $G$ is harmoniously matched. Moreover, if $K$ has a matched pair whose R-harmonious sequence also has product $1$, then so does $G$.
\end{theorem}

\begin{proof}
Let $(u_0,\dots,u_{n-1})$, $(u_1',\dots,u_{n-1}')$ be a matched pair in $K$, and apply Lemma~\ref{lem:lifting} with $H=N$, $a_r=u_r$ and $a_i'=u_i'$. The natural map $K\to G/N$ is an isomorphism, so~\ref{L1} holds because $u$ is harmonious in~$K$. Condition~\ref{L2} is part of Definition~\ref{def:matched}. For~\ref{L3}, the only term of $u$ in $N$ is the one in $N\cap K=\{1\}$. Condition~\ref{L4} holds because $u'$ orders $K\setminus\{1\}$ and has the same end terms as $u$. Finally, the consecutive products of $u$ are all of $K$ and those of $u'$ are all of $K\setminus\{1\}$, which is~\ref{L5}. The last statement follows from Remark~\ref{rem:product}.
\end{proof}

The work of Beals et al.~\cite{BGHJ} shows that every abelian group $K$ that is non-binary, not an elementary abelian $2$-group and not $\Z_3$ is harmoniously matched. (The product conditions are automatic for such $K$; see Remark~\ref{rem:def}.) Combining this with Theorem~\ref{thm:split} gives the following.

\begin{corollary}\label{cor:abelianK}
Let $G=K\ltimes N$, where $N$ has odd order and $K$ is abelian, non-binary, not an elementary abelian $2$-group and not $\Z_3$. Then $G$ is harmoniously matched.
\end{corollary}

By the Schur--Zassenhaus theorem (see~\cite{Robinson}), a normal subgroup $N$ with $\gcd(|N|,[G:N])=1$ always has a complement. This gives:

\begin{corollary}\label{cor:SZ}
Let $N$ be a normal subgroup of odd order of $G$ with $\gcd(|N|,[G:N])=1$. If $G/N$ is harmoniously matched, then so is $G$. In particular, if $G$ has a normal $2$-complement $N$ (for example, if $G$ is nilpotent) and a Sylow $2$-subgroup $P$ of $G$ is harmoniously matched, then $G$ is harmoniously matched.
\end{corollary}

For instance, if a Sylow $2$-subgroup $P$ of $G$ is abelian, non-cyclic, not elementary abelian, and satisfies $N_G(P)=C_G(P)$, then $G$ has a normal $2$-complement by Burnside's normal $p$-complement theorem (see~\cite{Robinson}), and $G$ is harmoniously matched.

\begin{corollary}\label{cor:odd3}
Let $G$ be a group of odd order with a normal $3$-complement $N$, and let $P\cong G/N$ be a Sylow $3$-subgroup. Then $G$ is harmoniously matched if $P$ is abelian of order at least $9$, or if $P$ is one of the two non-abelian groups of order~$27$.
\end{corollary}

\begin{proof}
Since $\gcd(|N|,3)=1$, we have $G=P\ltimes N$ by Schur--Zassenhaus. If $P$ is abelian of order at least $9$, it is harmoniously matched by~\cite{BGHJ} (or by Theorem~\ref{thm:odd} if $P$ is not elementary abelian). The two non-abelian groups of order $27$ are harmoniously matched by Table~\ref{tab:small}. Apply Theorem~\ref{thm:split}.
\end{proof}

\begin{example}\label{ex:27}
Let $G=\langle a,x,y\mid a^9=x^3=y^3=1,\ xy=yx,\ xax^{-1}=a^4,\ yay^{-1}=a\rangle=(\Z_3\times\Z_3)\ltimes\Z_9$. Then $[G,G]=\langle a^3\rangle\cong\Z_3$ and $G^{\ab}\cong\Z_3^3$, so Theorem~\ref{thm:odd} does not apply. Since $\Z_3\times\Z_3$ is harmoniously matched, $G$ is harmoniously matched by Corollary~\ref{cor:abelianK}.
\end{example}

\begin{example}\label{ex:Frob21}
The Frobenius group $\Z_7\rtimes\Z_3=\langle x,y\mid x^7=y^3=1,\ yxy^{-1}=x^2\rangle$ has abelianization $\Z_3$, so neither Theorem~\ref{thm:odd} nor Theorem~\ref{thm:split} applies. The group $\Z_7\rtimes\Z_3$ is harmoniously matched; Appendix~\ref{app:seq} gives a matched pair. Hence $(\Z_7\rtimes\Z_3)\ltimes N$ is harmoniously matched for every group $N$ of odd order.
\end{example}

\subsection{Dihedral and dicyclic groups}

If $n$ is odd, then $D_{2n}$ fails the Hall--Paige condition. If $n$ is even, it satisfies the condition. Write $D_{2n}=\langle r,s\mid r^n=s^2=(sr)^2=1\rangle$. If $n=dm$ with $m$ odd and $\gcd(d,m)=1$, then $\langle r^d\rangle\cong\Z_m$ is normal with complement $\langle r^m,s\rangle\cong D_{2d}$, so $D_{2n}\cong D_{2d}\ltimes\Z_m$. Similarly, for the dicyclic group $\Dic_{4n}=\langle a,b\mid a^{2n}=1,\ b^2=a^n,\ bab^{-1}=a^{-1}\rangle$ with $n=4m$ and $m$ odd, we have $\Dic_{16m}\cong\Dic_{16}\ltimes\Z_m$, with complement $\langle a^m,b\rangle$.

\begin{corollary}\label{cor:dihedral}
Let $n=dm$ with $d\in\{6,8,10,12,14,16\}$, $m$ odd and $\gcd(d,m)=1$. Then $D_{2n}$ is harmoniously matched. In particular $D_{2n}$ is harmoniously matched whenever $n\equiv8\pmod{16}$ or $n\equiv16\pmod{32}$. Also, $\Dic_{16m}$ is harmoniously matched for every odd $m$.
\end{corollary}

\begin{proof}
The groups $D_{12},D_{16},D_{20},D_{24},D_{28},D_{32}$ and $\Dic_{16}$ are harmoniously matched by Table~\ref{tab:small} (explicit pairs for $D_{12}$ and $D_{16}$ are in Appendix~\ref{app:seq}). Apply Theorem~\ref{thm:split}.
\end{proof}

\section{Small groups and open problems}\label{sec:small}

An exhaustive computer search over all orderings, together with~\cite{BGHJ} for the abelian groups $\Z_5$ and $\Z_7$, gives the following.

\begin{proposition}\label{prop:nonR}
Among the nontrivial groups of order at most $8$ that satisfy the Hall--Paige condition, exactly three are not R-harmonious: $\Z_3$, $D_8$ and $Q_8$. The group $D_8$ is harmonious. The group $Q_8$ is neither harmonious nor R-harmonious. The group $\Z_2^3$ is R-harmonious but not harmonious.
\end{proposition}

One expects that $\Z_3$ and $D_8$ are the only groups satisfying the Hall--Paige condition that fail to be R-harmonious. Table~\ref{tab:small} lists the non-abelian groups for which we found matched pairs by a randomized depth-first search. Every matched pair found was checked independently.

\begin{table}[ht]
\centering
\renewcommand{\arraystretch}{1.15}
\begin{tabular}{@{}lll@{}}
\toprule
order & group & status\\
\midrule
$8$ & $D_8$, $Q_8$ & not R-harmonious\\
$12$ & $D_{12}$, $A_4$ & harmoniously matched\\
$16$ & $D_{16}$, $\Dic_{16}$, $D_8\times\Z_2$, $\Z_4\rtimes\Z_4$ & harmoniously matched\\
$20$ & $D_{20}$ & harmoniously matched\\
$21$ & $\Z_7\rtimes\Z_3$ & harmoniously matched\\
$24$ & $D_{24}$, $S_4$, $\mathrm{SL}(2,3)$, $D_8\times\Z_3$, $Q_8\times\Z_3$ & harmoniously matched\\
$27$ & $\mathrm{Heis}(\Z_3)$, $\Z_9\rtimes\Z_3$ & harmoniously matched\\
$28$ & $D_{28}$ & harmoniously matched\\
$32$ & $D_{32}$ & harmoniously matched\\
\bottomrule
\end{tabular}
\medskip
\caption{Results of the computer search. The groups $D_8\times\Z_3$ and $Q_8\times\Z_3$ are harmoniously matched although $D_8$ and $Q_8$ are not R-harmonious.}
\label{tab:small}
\end{table}

Every group in Table~\ref{tab:small} can serve as the complement $K$ in Theorem~\ref{thm:split}. For example, $A_4\ltimes N$, $S_4\ltimes N$ and $\mathrm{SL}(2,3)\ltimes N$ are harmoniously matched for every group $N$ of odd order.

\begin{conjecture}\label{conj:main}
Every finite group satisfying the Hall--Paige condition is R-harmonious, except $\Z_3$, $D_8$ and $Q_8$.
\end{conjecture}

The remaining open cases for groups of odd order are the groups $G$ with $G^{\ab}\cong\Z_3^r$ that are not covered by Corollary~\ref{cor:odd3} or Theorem~\ref{thm:split}.

\begin{question}\label{q:Z3Z3}
Let $H$ be a normal subgroup of odd order of $G$ with $G/H\cong\Z_3\times\Z_3$, and suppose the extension does not split. Is $G$ harmoniously matched?
\end{question}

A positive answer would show that every group of odd order with $|G^{\ab}|>3$ is R-harmonious. Lemma~\ref{lem:lifting} reduces the question to finding coset representatives in $G$ that satisfy the identities~\ref{L2} and~\ref{L5} exactly. One natural approach is to obtain $a'$ from $a$ by deleting the identity and permuting segments of $a$ (without reversing them, since reversal does not preserve products in non-abelian groups), and to choose representatives so that the resulting identities hold in a free group. A computer search shows that no such ``formal'' rearrangement exists for $\Z_3\times\Z_3$, so a different idea is needed.

\begin{question}
Is $\Z_2^{\,r}\ltimes N$ R-harmonious for every $r\ge2$ and every group $N$ of odd order?
\end{question}

\begin{conjecture}
Matched integer pairs of length $n$ (Definition~\ref{def:intseq}) exist for every odd $n\ge5$.
\end{conjecture}

By Theorem~\ref{thm:intseq} and Appendix~\ref{app:seq}, the conjecture is open only for $n\equiv3\pmod{12}$ with $n\ge51$ and for $n\equiv33\pmod{36}$ with $n\ge69$. A positive answer to the following question would settle it.

\begin{question}\label{q:reversal}
Does every odd $n\ge7$ admit a harmonious sequence in $\Z_n$ satisfying the hypothesis of Lemma~\ref{lem:reversal}?
\end{question}

\appendix
\section{Explicit sequences}\label{app:seq}

Machine-readable versions of all sequences below, together with the matched pairs for the groups in Table~\ref{tab:small}, are contained in the accompanying file \texttt{verify\_Rharmonious.py}. That script also checks the steps of the proof of Theorem~\ref{thm:intseq} for all admissible $n<3000$.

\subsection*{Matched integer pairs for $n=15,27,33,39$}
These pairs were obtained from Lemma~\ref{lem:reversal}. In each case $k_1=0$, and $k'$ is obtained by deleting $k_1$ and reversing the block $k_2,\dots,k_j$.
\begin{itemize}
\item $n=15$ ($j=13$): $k=(-5,\allowbreak 0,\allowbreak 6,\allowbreak 5,\allowbreak -2,\allowbreak 7,\allowbreak 2,\allowbreak -1,\allowbreak -7,\allowbreak -6,\allowbreak 3,\allowbreak 1,\allowbreak -3,\allowbreak 11,\allowbreak -11)$.
\item $n=27$ ($j=9$): $k=(9,\allowbreak 0,\allowbreak 7,\allowbreak -9,\allowbreak 6,\allowbreak 10,\allowbreak -7,\allowbreak 11,\allowbreak 3,\allowbreak -2,\allowbreak 2,\allowbreak -8,\allowbreak -13,\allowbreak -12,\allowbreak 8,\allowbreak -3,\allowbreak -5,\allowbreak -10,\allowbreak -4,\allowbreak 12,\allowbreak -1,\allowbreak -11,\allowbreak 4,\allowbreak 13,\allowbreak 5,\allowbreak -6,\allowbreak 1)$.
\item $n=33$ ($j=11$): $k=(1,\allowbreak 0,\allowbreak -11,\allowbreak 16,\allowbreak -5,\allowbreak 14,\allowbreak -7,\allowbreak -8,\allowbreak 11,\allowbreak 8,\allowbreak 5,\allowbreak -12,\allowbreak 12,\allowbreak -15,\allowbreak 9,\allowbreak 15,\allowbreak 6,\allowbreak -16,\allowbreak -1,\allowbreak 13,\allowbreak -9,\allowbreak 7,\allowbreak 3,\allowbreak -4,\allowbreak -14,\allowbreak 10,\allowbreak -2,\allowbreak 4,\allowbreak 2,\allowbreak -10,\allowbreak -3,\allowbreak -13,\allowbreak -6)$.
\item $n=39$ ($j=3$): $k=(-7,\allowbreak 0,\allowbreak -3,\allowbreak 4,\allowbreak -4,\allowbreak 10,\allowbreak -11,\allowbreak -12,\allowbreak -2,\allowbreak -6,\allowbreak -19,\allowbreak 17,\allowbreak -13,\allowbreak -18,\allowbreak 9,\allowbreak 19,\allowbreak 7,\allowbreak 2,\allowbreak 11,\allowbreak -8,\allowbreak 18,\allowbreak 1,\allowbreak -16,\allowbreak 12,\allowbreak -10,\allowbreak 5,\allowbreak -17,\allowbreak -15,\allowbreak -9,\allowbreak 14,\allowbreak 3,\allowbreak 8,\allowbreak 15,\allowbreak 6,\allowbreak 16,\allowbreak 13,\allowbreak -1,\allowbreak -5,\allowbreak -14)$.
\end{itemize}

\subsection*{Matched pairs in groups}
In each case the first line is a harmonious sequence and the second line is an R-harmonious sequence. Both have product $1$, and they share the first and the last term.

\medskip\noindent $D_{12}=\langle r,s\mid r^6=s^2=(sr)^2=1\rangle$:
\begin{align*}
&r^3,\ s,\ r,\ r^5,\ r^2,\ r^4s,\ 1,\ r^4,\ r^3s,\ rs,\ r^2s,\ r^5s;\\
&r^3,\ rs,\ r^4s,\ r,\ r^4,\ r^2s,\ s,\ r^5,\ r^2,\ r^3s,\ r^5s.
\end{align*}

\noindent $D_{16}=\langle r,s\mid r^8=s^2=(sr)^2=1\rangle$:
\begin{align*}
&r^7s,\ r^7,\ r^6s,\ s,\ r^2,\ 1,\ r^4,\ r^3s,\ r,\ r^6,\ r^5s,\ r^2s,\ rs,\ r^4s,\ r^3,\ r^5;\\
&r^7s,\ r^2s,\ r^4s,\ rs,\ r^4,\ r^6s,\ r^6,\ r,\ r^3,\ s,\ r^7,\ r^2,\ r^5s,\ r^3s,\ r^5.
\end{align*}

\noindent $\Z_7\rtimes\Z_3=\langle x,y\mid x^7=y^3=1,\ yxy^{-1}=x^2\rangle$:
\begin{quote}\raggedright
$x^5y^2,\allowbreak\ x^6y,\allowbreak\ x^4y^2,\allowbreak\ x^6y^2,\allowbreak\ xy^2,\allowbreak\ x^2y,\allowbreak\ x^3y,\allowbreak\ x^5y,\allowbreak\ x^3,\allowbreak\ x,\allowbreak\ xy,\allowbreak\ x^4y,\allowbreak\ x^2,\allowbreak\ x^4,\allowbreak\ y^2,\allowbreak\ x^3y^2,\allowbreak\ 1,\allowbreak\ x^5,\allowbreak\ x^2y^2,\allowbreak\ x^6,\allowbreak\ y$;\\[2pt]
$x^5y^2,\allowbreak\ x^6y^2,\allowbreak\ x^3y^2,\allowbreak\ y^2,\allowbreak\ x^4y,\allowbreak\ x^5,\allowbreak\ xy,\allowbreak\ x^4,\allowbreak\ x^2,\allowbreak\ x^3y,\allowbreak\ x^5y,\allowbreak\ x^6y,\allowbreak\ xy^2,\allowbreak\ x^6,\allowbreak\ x^2y^2,\allowbreak\ x^3,\allowbreak\ x,\allowbreak\ x^4y^2,\allowbreak\ x^2y,\allowbreak\ y$.
\end{quote}

\noindent $\Z_9\rtimes\Z_3=\langle a,b\mid a^9=b^3=1,\ bab^{-1}=a^4\rangle$:
\begin{quote}\raggedright
$a^3b,\allowbreak\ a^5,\allowbreak\ 1,\allowbreak\ a^7b,\allowbreak\ a^7b^2,\allowbreak\ a^4,\allowbreak\ a^5b,\allowbreak\ a^6b,\allowbreak\ a^3b^2,\allowbreak\ a^5b^2,\allowbreak\ a,\allowbreak\ a^3,\allowbreak\ ab^2,\allowbreak\ b,\allowbreak\ a^4b^2,\allowbreak\ a^6,\allowbreak\ a^8b^2,\allowbreak\ b^2,\allowbreak\ a^6b^2,\allowbreak\ a^2b,\allowbreak\ a^2,\allowbreak\ ab,\allowbreak\ a^8b,\allowbreak\ a^8,\allowbreak\ a^7,\allowbreak\ a^2b^2,\allowbreak\ a^4b$;\\[2pt]
$a^3b,\allowbreak\ a^7b,\allowbreak\ a^7b^2,\allowbreak\ a^2b^2,\allowbreak\ a^8b^2,\allowbreak\ a^3,\allowbreak\ a^3b^2,\allowbreak\ a,\allowbreak\ a^6b,\allowbreak\ a^8,\allowbreak\ ab^2,\allowbreak\ b^2,\allowbreak\ a^2,\allowbreak\ a^4,\allowbreak\ a^4b^2,\allowbreak\ b,\allowbreak\ a^5b^2,\allowbreak\ a^2b,\allowbreak\ a^6,\allowbreak\ a^8b,\allowbreak\ ab,\allowbreak\ a^6b^2,\allowbreak\ a^5b,\allowbreak\ a^7,\allowbreak\ a^5,\allowbreak\ a^4b$.
\end{quote}


\begin{thebibliography}{99}

\bibitem{AKP} B.~Alspach, D.~L.~Kreher and A.~Pastine, The Friedlander--Gordon--Miller conjecture is true, \emph{Australas.\ J.\ Combin.} \textbf{67} (2017), 11--24.

\bibitem{BGHJ} R.~Beals, J.~A.~Gallian, P.~Headley and D.~Jungreis, Harmonious groups, \emph{J.\ Combin.\ Theory Ser.~A} \textbf{56} (1991), no.~2, 223--238.

\bibitem{BCCSZ} J.~N.~Bray, Q.~Cai, P.~J.~Cameron, P.~Spiga and H.~Zhang, The Hall--Paige conjecture, and synchronization for affine and diagonal groups, \emph{J.\ Algebra} \textbf{545} (2020), 27--42.

\bibitem{Evans} A.~B.~Evans, The admissibility of sporadic simple groups, \emph{J.\ Algebra} \textbf{321} (2009), no.~1, 105--116.

\bibitem{FT} W.~Feit and J.~G.~Thompson, Solvability of groups of odd order, \emph{Pacific J.\ Math.} \textbf{13} (1963), 775--1029.

\bibitem{FGM} R.~J.~Friedlander, B.~Gordon and M.~D.~Miller, On a group sequencing problem of Ringel, \emph{Congr.\ Numer.} \textbf{21} (1978), 307--321.

\bibitem{HP} M.~Hall and L.~J.~Paige, Complete mappings of finite groups, \emph{Pacific J.\ Math.} \textbf{5} (1955), 541--549.

\bibitem{JavSym} M.~Javaheri, Symmetric harmoniousness of odd-order groups, \emph{Australas.\ J.\ Combin.} \textbf{96} (2026), no.~1, 134--141.

\bibitem{MP} A.~M\"uyesser and A.~Pokrovskiy, A random Hall--Paige conjecture, \emph{Invent.\ Math.} \textbf{240} (2025), 779--867.

\bibitem{Ollis} M.~A.~Ollis, Sequenceable groups and related topics, \emph{Electron.\ J.\ Combin.}, Dynamic Survey DS10.

\bibitem{Paige} L.~J.~Paige, Complete mappings of finite groups, \emph{Pacific J.\ Math.} \textbf{1} (1951), 111--116.

\bibitem{Robinson} D.~J.~S.~Robinson, \emph{A Course in the Theory of Groups}, 2nd ed., Graduate Texts in Mathematics~80, Springer, New York, 1996.

\bibitem{Wang} C.-D.~Wang, On the harmoniousness of dicyclic groups, \emph{Discrete Math.} \textbf{120} (1993), no.~1--3, 221--225.

\bibitem{WangLeonard} C.-D.~Wang and P.~A.~Leonard, More on sequences in groups, \emph{Australas.\ J.\ Combin.} \textbf{21} (2000), 187--196.

\bibitem{Wilcox} S.~Wilcox, Reduction of the Hall--Paige conjecture to sporadic simple groups, \emph{J.\ Algebra} \textbf{321} (2009), no.~5, 1407--1428.

\end{thebibliography}
\end{document}